\documentclass[11pt]{amsart}
\usepackage{amssymb,mathrsfs,graphicx}
\usepackage{hyperref,color}
\usepackage{amsmath}
\usepackage{amsfonts}
\usepackage{upref}
\usepackage{float}
\usepackage{xfrac}
\usepackage{amssymb}

\newtheorem{theorem}{Theorem}[section]
\newtheorem{lemma}[theorem]{Lemma}

\newtheorem{remark}[theorem]{Remark}

\title[Dissipation vs quadratic nonlinearity]{Dissipation vs. quadratic nonlinearity:\\
from {\em a priori} energy bound to\\higher-order regularizing effect}

\author[Animikh Biswas]{Animikh Biswas}
\address[Animikh Biswas]{\newline
  Department of Mathematics and Statistics\newline
	      University of Maryland, Baltimore County, Baltimore, MD 21250 USA}
		\email[]{abiswas@umbc.edu}
\author[Eitan Tadmor]{Eitan Tadmor}
\address[Eitan Tadmor]{\newline
  Center for Scientific Computation And Mathematical Modeling (CSCAMM)\newline
   and \newline
  Department of Mathematics,  Institute for Physical Science and Technology\newline
  University of Maryland, 
  College Park, MD 20742 USA}
\email[]{tadmor@cscamm.umd.edu}
\urladdr{http://www.cscamm.umd.edu/tadmor}
\thanks{This research was supported in part by NSF grants DMS11-09532 (A. Biswas) and 
DMS10-08397 and NSF KI-Net grant RNMS11-07444 (E. Tadmor).}
\subjclass{35K55,35B65,35Qxx}
\keywords{Parabolic equations, quadratic nonlinearity, energy functional, smoothness, Navier-Stokes equations, Keller-Segel equation, QG equation}

\newcommand{\comments}[1]{}

\newcommand{\eps}{\varepsilon}
\renewcommand{\phi}{\varphi}
\newcommand{\ra}{\rightarrow }
\newcommand{\h}{\mathbb H}

\newcommand{\g}{\gamma}
\newcommand{\hh}{\dot{\h}}
\newcommand{\R}{\mathbb R}
\newcommand{\p}{\mathbb P}
\newcommand{\N}{\mathbb N}
\newcommand{\nn}{\nonumber}
\newcommand{\D}{\displaystyle }

\newcommand{\vt}{\vartheta}

\newcommand{\cal}{\mathcal }

\newcommand{\dt}{\frac{d}{dt}}

\newcommand{\V}{{\cal \vee}}
\newcommand{\sL}{{\cal L}}

\def\intav#1{\mathchoice
          {\mathop{\vrule width 6pt height 3 pt depth -2.5pt
                  \kern -9pt \intop}\nolimits_{\kern -6pt#1}}%
          {\mathop{\vrule width 5pt height 3 pt depth -2.6pt
                  \kern -6pt \intop}\nolimits_{#1}}%
          {\mathop{\vrule width 5pt height 3 pt depth -2.6pt
                  \kern -6pt \intop}\nolimits_{#1}}%
          {\mathop{\vrule width 5pt height 3 pt depth -2.6pt
                  \kern -6pt \intop}\nolimits_{#1}}}

\newcommand{\charfn}[1]{{\raisebox{1.2pt}{\mbox{$\chi
_{\kern-1pt\lower3pt\hbox{{$\scriptstyle{#1}$}}}$}}}}

\newcommand{\be}{\begin{equation}}
\newcommand{\ee}{\end{equation}}
\newcommand{\bes}{\begin{equation*}}
\newcommand{\ees}{\end{equation*}}
\newcommand{\bea}{\begin{eqnarray}}
\newcommand{\eea}{\end{eqnarray}}
\newcommand{\beas}{\begin{eqnarray*}}
\newcommand{\eeas}{\end{eqnarray*}}
\newcommand{\di}{\partial}
\newcommand{\ff}{{\hh^{-\frac{1}{2}}}}
\newcommand{\eop}{$\square$}
\newcommand{\bc}{\beta_{c}}
\newcommand{\br}{\beta_{R}}
\newcommand{\bt}{\beta_{T}}
\newcommand{\bs}{\beta_{S}}
\newcommand{\ffc}{\frac{\beta_c}{\vt}}
\newcommand{\ffr}{\frac{\beta_R}{\vt}}

\begin{document}
\date{\today}
\begin{abstract}
We consider a rather general class of  evolutionary  PDEs involving dissipation (of possibly fractional order), which competes  with quadratic nonlinearities on the regularity of the overall equation. This includes as prototype models,  Burgers' equation, the Navier-Stokes equations, the surface quasi-geostrophic equations  and the Keller-Segel model for chemotaxis.   Here we establish a Petrowsky type parabolic estimate  of such equations which entail a precise time decay  of higher-order Sobolev norms for this class of equations. To this end, we introduce as a main new tool, an ``infinite order energy functional'', ${\mathcal E}(t): = \sum_n \alpha_n t^n \|(-\Delta)^{n\theta/2}u(\cdot,t)\|_{\hh^{\beta_c}}$ for appropriate 
critical regularity index $\beta_c$. 
 It captures the regularizing effect of \emph{all} higher order derivatives of $u(\cdot,t)$, by proving --- for a carefully, problem-dependent choice of weights $\{\alpha_n\}$, that  ${\mathcal E}(t)$ is non-increasing in time.

\end{abstract}

\maketitle
\tableofcontents
\section{Introduction}\label{sec:1}

Consider a linear evolution equation 
 \begin{gather}   \label{lineveq}
 u_t + Au =0, \qquad u(0)=u_0 \in L^2(\R^d),
 \end{gather}
where  $A:= (-\Delta)^{\vt}$.
\comments{ It  is well-known that this generates a contractive semi-group on $L^2(\R^d)$. This is due to the fact that
 the ``energy" $E(t) = \|u(t)\|^2$ is a non-increasing function of $t$.}
 It is  well-known that   higher  Sobolev norms obey  the decay estimate\footnote{Throughout the paper, we use the $L^2$-norm, $\|\cdot\|$, and we let  $\hh^s, s \in \R$ denote  the homogeneous $L^2$-based Sobolev (potential)  spaces  
$\displaystyle 
\hh^s := \{u \in {\cal S}'(\R^d): \ \ \|u\|_{\dot{\h}^s}:=\|(-\Delta)^{s/2} u\| < \infty\}.$
}
 \begin{gather}  \label{linest}
\|u\|_{\hh^m}^2:= \|(-\Delta)^{m/2}u(t)\|^2 \le \frac{c_m}{t^{m/\vt}}\|u_{0}\|^2\ \quad \mbox{for all}\  t>0.
 \end{gather}
 In fact, (\ref{lineveq})
 is said to be parabolic of order $\vt$ in the sense of Petrowsky \cite{taylor}, if the estimate above holds. 
 The inequality (\ref{linest}) provides both a decay estimate for the higher Sobolev (semi-)norms for large times,
  as well as a regularizing effect for $L^2$ initial data.
  The usual proof of (\ref{linest}) involves Fourier analysis:  
  observing that $\widehat{u}(\xi,t)=e^{-t|\xi|^{2\vt}}\widehat{u}_0(\xi)$,  one obtains
 \begin{gather} \label{bootstrapest}
  \|u(t)\|_{\dot{\h}^{n\theta}}^2  = \int |\xi|^{2n\vt}e^{-2t|\xi|^{2\vt}}|\widehat{u}_0(\xi)|^2\, d\xi
  \le \frac{c_n}{t^n}\int |\widehat{u}_0(\xi)|^2\, d\xi = \frac{c_n}{t^n}\|u_0\|^2.
  \end{gather}
We illustrate a new bootstrap procedure to derive (\ref{linest}) which avoids the use of the Fourier transform, and subsequently, will be generalized to a much larger class of dissipative equations with quadratic nonlinearities.
Set  $\displaystyle \Lambda = (-\Delta)^{1/2}$ as  the self-adjoint root of the minus Laplacian, so that equation \eqref{lineveq}  reads
\begin{gather}\label{fracheat}
u_t= -\Lambda^{2\theta}u.
\end{gather}
Let $\{\alpha_n\ge 0\}_{n \in \N}$  be a sequence to be determined shortly. ``Integrating" the equation in its form \eqref{fracheat} against $\Lambda^{n\theta}u$ yields
\begin{gather*}
2\alpha_n t^n (\Lambda^{n\theta} u,\Lambda^{n\theta} u_t) =
2\alpha_nt^n(\Lambda^{n\theta} u , - \Lambda^{n\theta}\Lambda^{2\theta}u)
= -2\alpha_nt^n\|\Lambda^{(n+1)\vt} u\|^2,
\end{gather*}
and hence 
\begin{gather}  \label{linrecur}
\dt \left[ \alpha_n t^n \|\Lambda^{n\theta}u\|^2\right]
= \left\{\begin{array}{ll}
- 2\alpha_0\|\Lambda^{\vt}u\|^2, & n=0, \\ \\
n \alpha_nt^{n-1}\|\Lambda^{n\theta}u\|^2 - 2\alpha_nt^n\|\Lambda^{(n+1)\vt}u\|^2, & n\geq1.
\end{array}
\right.
\end{gather}
Set $\alpha_0=1$. If we now choose the $\alpha$'s recursively, $n\alpha_n = 2\alpha_{n-1}$,
  then the expression  on the right hand side of (\ref{linrecur}) amounts to a \emph{telescoping sum} and we end up with
\begin{gather*}
\dt \left(\sum_{n=0}^{\infty} \frac{2^n}{n!} t^n \|\Lambda^{n\theta}u(\cdot,t)\|^2\right) = 0.
\end{gather*}
We conclude that the \emph{infinite order energy functional}\footnote{Unless otherwise stated, we suppress the spatial dependence of $u(\cdot,t)$ and we only specify the time dependence of the various energy norms.}
\[
{\cal E}(t)= \sum_{n=0}^{\infty}  \alpha_nt^n \|\Lambda^{n\theta}u(t)\|^2, \qquad \alpha_n=\frac{2^n}{n!}, n\ge 0,
\]
 is conserved over time.
As a corollary, we recover the estimate (\ref{bootstrapest}), 
\begin{gather*}
\|u\|_{\hh^{n\theta}}^2 \le \frac{c_n}{t^n}\|u_0\|^2, \qquad c_n=\frac{1}{\alpha_n} =\frac{n!}{2^n}.
\end{gather*}
Note that the same result holds, with identical proof,  if in the definition of the infinite order energy functional, the $L^2$ norm $\|\cdot\|$ is replaced by any homogeneous Sobolev norm $\|\cdot \|_{\hh^{\beta}}, \beta \in \R$ (see  
\eqref{hbd} below).

\medskip
  In this paper, we consider a class of  nonlinear dissipative evolution equations of the form
\begin{subequations}\label{quadeqn}
\begin{equation}
u_t + Au = B(u,u),
\end{equation}
where $A$ is a dissipative operator of order $2\theta$
\begin{equation}\label{eq:dissp}
A:=(-\Delta)^\vt\, 
\end{equation}
and $B(\cdot,\cdot)$ is a bilinear operator of the form
\begin{equation}\label{eq:quadn}
B(u , v) := R(Su \otimes Tv).
\end{equation}
\end{subequations}
Here $R, S, T$ are Fourier multipliers of homogeneous degree $\beta_R, \beta_S$ and $\beta_T$ respectively.
These types of nonlinearities are often encountered in many models   in physics and biology, including the prototypical examples of  Burgers'
equation, Navier-Stokes equations, the surface quasi-geostrophic
equation and the Keller-Segel model for chemotaxis.

In some of the above mentioned examples, the nonlinearity satisfies the skew-symmetry $(B(u,v),v)=0$ which in turn implies that $\|u(t)\|$ is non-increasing.
Our goal here is to show that the same non-increasing property holds for an appropriately defined infinite order energy functional which contains all the higher order derivatives of the solution. As a corollary, 
we show that the regularizing effect of the dissipative term $Au$ in (\ref{eq:dissp}) 
balances the loss of regularity due to the quadratic nonlinearity $B(u,u)$ in (\ref{eq:quadn}) and
 the Petrowsky type estimate (\ref{bootstrapest}) still holds. 

 In order to do this, 
we introduce
the  ``infinite order energy functional'' 
\begin{subequations}\label{eqs:higherenergy}
\begin{gather}  \label{higherenergy}
{\cal E}(t) := \sum_{n=0}^{\infty} \alpha_nt^n\|(-\Delta)^{n\vt/2}u(t)\|_{\hh^{\beta_c}}^2, \quad \alpha_0=1;
\end{gather}
here  
\begin{equation}
\beta_c:=\beta_R+\beta_S+\beta_T+\frac{d}{2}-2\vt, 
\end{equation}
\end{subequations}
is the order of ``critical regularity'' which balances the dissipation (\ref{eq:dissp}) vs. the quadratic nonlinearity (\ref{eq:quadn}). Thus, for example, in the typical cases of Burgers and Navier-Stokes equations where $\beta_R=\vt=1$ and  $\beta_S=\beta_T=0$, we find the (usual) critical regularity space of order $\beta_c=\sfrac{d\!}{2}-1$. 
The functional ${\cal E}(\cdot)$ contains  appropriately weighted sum of {\it all} the higher order derivatives of $u$; the choice of the weights $\{\alpha_n\}$ is problem dependent.
Our main result, Theorem \ref{thm:gensmalldata}, shows  that even in the (rather general) nonlinear setting of (\ref{quadeqn}), 
there exists a proper choice of $\{\alpha_n\}$ such that the corresponding functional
${\cal E}(t)$ is  non-increasing in time provided  $\|u_0\|_{\hh^{\beta_c}}$ is sufficiently small.  
This immediately yields Petrowsky type estimates  of the type (\ref{linest}), namely,
\begin{gather}  \label{hbd}
\|u(t)\|_{\hh^{n\vt+\beta_c}}^2 :=\|(-\Delta )^{n\vt/2}u(t)\|_{\hh^{\beta_c}}^2 \le \frac{1}{\alpha_nt^n}\|u_0\|_{\hh^{\beta_c}}^2, \quad n \ge 1.
\end{gather}
Note that the restriction to ``small data" is necessary due to the rather general form  of \eqref{quadeqn}; it is well-known that the 2D Keller-Segel model  for example, corresponding to $(R,S,T)\mapsto(\nabla_x, I, \nabla_x\Delta^{-1})$, ``blows up"  if $\|u_0\|$ is sufficiently large, \cite{per}.

In certain applications it may be more appealing to use  an $L^2$-based infinite order energy functional instead of the Sobolev-based energy $\|\cdot\|_{\dot{\h}^{\beta_c}}$ in \eqref{higherenergy},  since the former is intimately  related to a  ``physical energy".  
In this case, the  higher order bounds  (\ref{hbd}) follow from a (single) lower-order decay, consult Theorem \ref{thm:main}. In particular, as noted in Remark \ref{rem:largetime} below, the $L^2$-balance induced by skre-symmetric $B(\cdot,\cdot)$'s, implies the higher-order decay \eqref{hbd}, at least for large enough time, $t>t_0>0$. In fact, our method shows that for large time, ${\cal E}(t)=O(\|u(t)\|^2)$.
This observation provides a significant advantage when there is exponential time decay of $\|u(t)\|$ (e.g., in the periodic setting and for certain classes of initial data in the whole space \cite{S3}): one can then leverage the similar begavior of ${\cal E}(t)$ to conclude  that  the higher order Sobolev norms of $u(t)$
  decay at exactly the same exponential rate as $\|u(t)\|$ does. 
  
As examples for the versatility of our approach, we pursue the specific examples of  the one-dimensional Burgers' equation, the two- and three-dimensional Navier-Stokes equations, 
 the two-dimensional surface quasi-geostrophic equations and the two- and three-dimensional Keller-Segel model of chemotaxis. 
Higher order decay results for these equations have been previously obtained in many specific   setups and we mention here \cite{S5, SW, Ti} for the Navier-Stokes equations,   and \cite{CW, dong, dong1},  for the surface quasi-geostrophic equations. More references are found in section \ref{sec:4}. 
  Indeed, there is a host of optimal decay results available in the literature for these equations which employ different strategies to derive optimal decay rates under  different  structural assumptions; a complete list of references will be too long to be quoted here.
  We emphasize that our main focus, however, lies in the  new approach  based on the use of an infinite-order
energy functional: since it is independent  of Fourier-based arguments, the proposed approach enables us to pursue the same \emph{unified} framework for analyzing  the time decay of  the large  class of dissipative equations with quadratic nonlinearities outlined above.
 
The organization of the paper is as follows.  In Section \ref{sec:2} we illustrate our basic technique on the Burgers' equation, and in Section \ref{sec:3} we provide a general formulation of our result from which all our applications follow. Section \ref{sec:4} is devoted to the applications of our main results
  to the case of the Navier-Stokes equations, the 2D surface quasi-geostrophic equations and the Keller-Segel model of chemotaxis while Section \ref{sec:5} we provide the proofs of our results. In the Appendix, for completeness, we give details of an existence theorem the particular cases of which, such as the Navier-Stokes and the Keller-Segel model of chemotaxis, are well-known.

\section{Burgers' equation: a warm-up for the nonlinear case}\label{sec:2}
Here, we show  how to adapt the real-space approach of Section \ref{sec:1}  to the nonlinear setting. 
We illustrate the general method in the context of the  one-dimensional viscous Burgers'  equation,
\be  \label{1dburger}
u_t + (u^2)_x =  u_{xx}, \qquad (x,t)\in \R\times \R_+.
\ee

\begin{theorem}  \label{thm:1dburger}
Let $u$ be a solution of Burgers' equation  \eqref{1dburger} subject to initial data $u_0 \in  L^2(\R)$, 
such that $\|u_0\|_{\hh^{-\frac{1}{2}}}$ is sufficiently small. 
Then,  there exists an adequate choice of constants $\{\alpha_n >0\}_{n\in \N}$ (depending on $\|u_0\|_{\ff}$), such that the infinite order energy functional,
\begin{equation}\label{eq:burgE}
{\cal E}(t) 
:= \sum_{n=0}^{\infty} \alpha_nt^n\|\Lambda^nu(t)\|_\ff^2, \quad \alpha_0=1, 
\end{equation}
is non-increasing for all $t>0$, and in particular, 
\begin{equation}\label{eq:burgdecay}
\|u(t)\|_{\dot{\h}^{n-\sfrac{1}{2}}} \le \frac{1}{\alpha_n t^n}\|u_0\|_\ff^2.
\end{equation}
\end{theorem}

\noindent
\begin{proof}
We first note, consult Theorem \ref{thm:localexistence} below, that Burgers' equation \eqref{1dburger} admits a {\em mild} solution, $u(\cdot) \in C([0,\infty);\ff) \cap L^{\infty}_{loc}((0,\infty);L^2(\R))$.
 Here, and in all subsequent results, we will provide formal {\em a priori} estimates which can be made rigorous in the usual manner by establishing uniform bounds on smooth approximate solutions and then passing to the limit.\newline
We begin the proof with the following lemma. 
\begin{lemma}   \label{lem:1dburger}
Let $u$ be a solution of \eqref{1dburger} and assume that $\|u_0\|_\ff$ is sufficiently small. Then,
 for all $t>0, {\D \dt \|u(t)\|_\ff^2 \le - \|\Lambda^{1/2} u(t)\|^2}$ and in particular
\[
\|u(t)\|_\ff^2 \le \|u_0\|_\ff^2.
\]
\end{lemma}

\noindent
Indeed, ``pairing''  \eqref{1dburger} against $\Lambda^{-1} u$, we obtain that    $\|u\|_\ff^2\equiv (u, \Lambda^{-1}u)$ satisfies
\[
\frac{1}{2} \dt \|u(t)\|_\ff^2 + \|\Lambda^{1/2} u(t)\|^2 = - (\di_x(u^2),\Lambda^{-1}u)
=- (\Lambda^{-1}\di_x(u^2),u).
\]
Using the $L^2$-boundedness of the Hilbert transform $\Lambda^{-1}\di_x$ followed by Sobolev's bound
$\|u\|_{L^4} \lesssim \|\Lambda^{1/4}u\|$ and interpolation, we obtain
\[
|(\Lambda^{-1}\di_x(u^2),u)| \le \|u^2\|\|u\|=\|u\|_{L^4}^2\|u\|
\lesssim \|\Lambda^{1/4}u\|^2\|u\| \leq C \|u\|_\ff \|\Lambda^{1/2}u\|^2.
\]
Consequently, $\displaystyle \frac{1}{2} \dt \|u(t)\|_\ff^2 + \|\Lambda^{1/2} u(t)\|^2(1- C\|u\|_\ff) \le 0$.
Thus, if $\|u_0\| < \sfrac{1}{2C}$, then $\|u(t)\|_\ff^2$ is non-increasing and the lemma follows.\newline
We note in passing that the smallness assumption of $\|u_0\|_\ff$ was required just in order to insure that $\|u(t)\|_\ff$ is non-increasing:  granted that bound of $\|u(t)\|_\ff$, we continue with the proof of Theorem \ref{thm:1dburger}. 

\medskip
Pairing \eqref{1dburger} with $\Lambda^{2n-1}u$  we  obtain 

\begin{align}  \label{derivburgera}
\dt \left[\alpha_n t^n\|\Lambda^{n-\frac12}u(t)\|^2\right]  
 \ \left\{ \begin{array}{ll}
\le - \|\Lambda^{\frac12} u\|^2, & n=0,\\ \\
= \overbrace{n\alpha_nt^{n-1}\|\Lambda^{n-\frac12} u\|^2 - 2\alpha_nt^n\|\Lambda^{n+\frac12}u\|^2}^{\text{dissipation -- telescoping sum} \eqref{linrecur}}  \\ \\
\quad -\overbrace{2 \alpha_n t^n (\Lambda^{n-\frac12}u, \Lambda^{n-\frac12}\di_x(u^2))}^{\text{nonlinearity}},&  n\ge 1.
 \end{array}
 \right.
 \end{align}
The case $n=0$ with $\alpha_0=1$ is just lemma \ref{lem:1dburger}; the remaining cases of $n\geq1$ require to  bound the nonlinearity in the third-term on the right, so that it can be ``absorbed" into the carefully tuned dissipative telescoping sum. 
To this end, we recall the Kato-Ponce inequality, \cite{K,GS},
\begin{equation} 
\|\Lambda^\beta (vw)\|_{L^2} 
 \le C_n\left(\|\Lambda^\beta v\|_{L^{p_1}}\|w\|_{L^{q_1}}+\|v\|_{L^{p_2}}\|\Lambda^\beta w\|_{L^{q_2}}\right), \quad \frac{1}{p_i}+\frac{1}{q_i}=\frac{1}{2}. \label{katoponce}
\end{equation}
Using this with $\beta =n-\frac12,\, p_i=q_i=4$ together with the Sobolev inequality $\|z\|_{L^4} \lesssim \|\Lambda^{1/4}z\|$ and followed by a straightforward interpolation of $\|\Lambda^{n-\frac14}u\|$ in terms of $\|\Lambda^{n\pm\sfrac{1}{2}}u\|$, yields
\be   \label{kpapply}
\|\Lambda^{n-\frac12}(u^2)\|\le C_n\|\Lambda^{n-\frac14}u\|\|\Lambda^{\frac14}u\|
\equiv C_n\|\Lambda^{n-\frac12}u\|^{\sfrac{3\!}{4}}\|\Lambda^{n+\frac12}u\|^{\sfrac{1\!}{4}}\|\Lambda^{\frac14}u\|.
\ee
The last bound, \eqref{kpapply}, followed by Young's inequality imply that the third term on the right of \eqref{derivburgera} does not exceed
\begin{eqnarray*}
2\alpha_nt^n |(\di_x\Lambda^{n-\frac12}u,\Lambda^{n-\frac12} (u^2))| 
& \lesssim & 2\alpha_nt^n \|\Lambda^{n+\frac12}u\|\|\Lambda^{n-\frac12}(u^2 )\| \\
& \le & 2\alpha_nt^n C_n \|\Lambda^{\frac14}u\|\|\Lambda^{n+\frac12}u\|^{\sfrac{5\!}{4}}\|\Lambda^{n-\frac12} u\|^{\sfrac{3\!}{4}} \\
& \le & 2\alpha_nt^n\left(\frac{1}{\sfrac{8\!}{5}}\|\Lambda^{n+\frac12}u\|^2
+ \frac{1}{\sfrac{8\!}{3}}\left(C_n \|\Lambda^\frac14u\|\right)^{\sfrac{8\!}{3}}\|\Lambda^{n-\frac12}u\|^2\right).
\end{eqnarray*}
Inserting this back into into \eqref{derivburgera} we end up with the recursive estimate,
\begin{equation}\label{eq:clo}
 \dt \left[\alpha_n t^n\|\Lambda^{n-\frac12}u(t)\|^2\right]\le 
   \alpha_nt^{n-1}\!\left(\!n+ \frac34 C^{\sfrac{8\!}{3}}_n t\|\Lambda^\frac14u\|^{\sfrac{8\!}{3}}\right)\!\|\Lambda^{n-\frac12}u\|^2 - \frac34 \alpha_nt^n\|\Lambda^{n+\frac12}u\|^2. 
\end{equation} 
We now come to the heart of matter -- a closure of the recursive  bounds  in \eqref{eq:clo}. A straightforward interpolation bound 
$\displaystyle 
\|\Lambda^{\frac14}u\| \lesssim \|u\|^{\sfrac{1\!}{4}}_{\dot{\h}^{-\frac12}}\|\Lambda^\frac12 u\|^{\sfrac{3\!}{4}}$ implies that (recalling $ {\cal E}(t) \sim \alpha_1 t\|\Lambda^\frac12 u\|^2 + \ldots$), 
\be\label{intersmall}
t\|\Lambda^{\frac14}u\|^{\sfrac{8\!}{3}} \lesssim \|u_0\|^{\sfrac{2\!}{3}}_{\dot{\h}^{-\frac12}} t \|\Lambda^\frac12 u(t)\|^2 \le C \|u_0\|^{\sfrac{2\!}{3}}_{\dot{\h}^{-\frac12}} \frac{1}{\alpha_1}{\cal E}(t).
\ee
The $\alpha_n$'s will be chosen so that ${\cal E}(t)$ is decreasing, and in particular, ${\cal E}(t) \leq 
\|u_0\|_\ff^2$.
Thus, starting with $\alpha_0=1$, and choosing the $\alpha_n$'s recursively
\[
\alpha_n\left(n+\frac34 C_n^{\sfrac{8\!}{3}}\frac{C}{\alpha_1}\|u_0\|_{\ff}^{\sfrac{8\!}{3}}\right)= \frac34\alpha_{n-1}, \quad n=1,2,\ldots,
\]
we end up with a telescoping sum in \eqref{eq:clo}  
\begin{eqnarray}
\lefteqn{\quad \dt {\cal E}(t)  \leq  -\|\Lambda^\frac12 u\|^2} \\
& &  \quad + \sum_{n=1}^\infty \alpha_nt^{n-1}\overbrace{\left(n+ \frac34 C^{\sfrac{8\!}{3}}_n \frac{C}{\alpha_1}\|u_0\|_\ff^{\sfrac{2\!}{3}}{\cal E}(t)\right)}^{\le \frac34 \alpha_{n-1}/\alpha_n} \|\Lambda^{n-\frac12}u\|^2 - \frac34 \alpha_nt^n\|\Lambda^{n+\frac12}u\|^2 \le 0, \nonumber
\end{eqnarray}
and the result (\ref{eq:burgE}) follows. 
\end{proof} 

\begin{remark} 
{\em Observe that a key role of the proof lies in the closure \eqref{intersmall} where $t\|\Lambda^\frac14 u\|^{\sfrac{8\!}{3}}$
is upper-bounded by $t\|\Lambda^\frac12 u\|^2 \lesssim {\cal E}(t)$. The type of  a closure   argument will be pursued in a more general setup below, when interpolation with higher-order $\Lambda^s u$ will be closed with an infinite-order energy functional ${\cal E(t)}$.}
\end{remark}

For some applications, it may be more appealing to use the $L^2$-norm for the infinite order energy functional since it represents physical ``energy". This can be done provided one makes adequate assumptions on the decay of $L^2$ norm: by (interpolation of) (\ref{eq:burgdecay}), the $L^2$ decay sought is of order $\|u(t)\| \lesssim t^{-1/4}$. This follows from Theorem \ref{thm:1dburger} when $\|u_0\|_{\hh^{-\frac 12}}$ is sufficiently small and is in agreement with the $L^2$-decay of Burgers' solution for general $\in L^1 \cap L^2$-initial data, \cite{W}. We have the following result.
\begin{theorem}  \label{thm:1dburgerl2}
Let $u$ be a solution of the Burgers' equation \eqref{1dburger} subject to $L^2$-initial data $u_0$, and assume it satisfies the following  $L^2$-decay  --- there exists a constant  possibly dependent on the initial data, $D_0=D(u_0)$,   such that 
\be  \label{schonbek1dburger}
\|u(t)\| \le \frac{D_0}{t^{1/4}}, \qquad \forall\ t>0.
\ee
Then, the infinite order energy functional, 
\begin{gather*}
{\cal E}(t) :=\sum_{n=0}^{\infty} \alpha_nt^n\|(-\Delta)^{n/2}u(t)\|^2, \qquad \alpha_0=1
\end{gather*}
with  $\alpha_n$  defined recursively in terms of the Kato-Ponce constants $C_n$'s in \eqref{kpapply},  
\be \label{critalphaineq}
\alpha_n := \frac12 \dfrac{\alpha_{n-1}}{C_n^{4}(1+D_0)^4}, \qquad n=1,2, \cdots ,
\ee
is non-increasing in time. In particular,  the high-order decay estimate follows  
\[
\|u(t)\|_{\dot{\h}^n}^2 \le \frac{1}{\alpha_nt^n}\|u_0\|^2.
\]
\end{theorem}
Thus, theorem \ref{thm:1dburgerl2} shows that $L^2$-decay implies higher-order regularity and faster decay, a theme that will repeat itself in our examples below.

\begin{proof} 
The proof of Theorem \ref{thm:1dburgerl2} closely resembles Theorem \ref{thm:1dburger}, the only difference being in the way interpolation inequality is used.
 From \eqref{1dburger}, we have
 \begin{align}  \label{derivburgerb}
 \lefteqn{\dt \left[\alpha_n t^n\|\Lambda^nu(t)\|^2\right]= } \nn \\
 & \qquad \left\{ \begin{array}{ll}
 - 2\alpha_0\|\Lambda u\|^2, & n=0,\\ \\
 n\alpha_nt^{n-1}\|\Lambda^n u\|^2 - 2\alpha_nt^n\|\Lambda^{n+1}u\|^2 - 
2 \alpha_n t^n (\Lambda^nu, \Lambda^n\di_x(u^2)),&  n\ge 1.
 \end{array}
 \right.
 \end{align}
As before, using \eqref{katoponce}, Sobolev and interpolation inequalities, we obtain
\begin{gather}  \label{finalest}
|(\Lambda^nu, \Lambda^{n}\di_x(u^2))| \le C_n\|u\|\|\Lambda^{n+1}u\|^{3/2}\|\Lambda^nu\|^{1/2}.
\end{gather}
Combining (\ref{derivburgerb}) and (\ref{finalest}), we   obtain for $n\geq 1$
\begin{subequations}\label{eqs:abc}
\begin{eqnarray}
\lefteqn{\dt \left[\alpha_n t^n\|\Lambda^nu(t)\|^2\right]} \nonumber \\
& & \le n \alpha_nt^{n-1}\|\Lambda^nu\|^2 - 2\alpha_nt^n\|\Lambda^{n+1}u\|^2
+2\alpha_nt^nC_n\|u\|\|\Lambda^{n+1}u\|^{3/2}\|\Lambda^nu\|^{1/2} \nn \nonumber \\
& & \le  n \alpha_nt^{n-1}\|\Lambda^nu\|^2 - \alpha_nt^n\|\Lambda^{n+1}u\|^2 + \alpha_nt^nC_n^{4}\|u\|^4\|\Lambda^nu\|^2 \label{first} \\
& & \le C_n^{4}\alpha_nt^{n-1}\|\Lambda^nu\|^2(1+D_0^4)- \alpha_nt^n\|\Lambda^{n+1}u\|^2; \label{second}
\end{eqnarray}
\end{subequations}
here, \eqref{first} follows from Young's inequality, and  \eqref{second} follows from   (\ref{schonbek1dburger}).  
Consequently,  our choice of  $\alpha_n$ in \eqref{critalphaineq} amount to a telescoping sum in \eqref{eqs:abc},
\begin{gather*}
 \dt \left(\sum_{n=0}^{\infty} \alpha_n t^n \|\Lambda^nu(t)\|^2\right) \le 
-\frac12 \sum_{n=0}^\infty \alpha_nt^n\|\Lambda^{n+1}u(t)\|^2 <0. 
 \end{gather*}
 This concludes the proof of the theorem. 
\end{proof} 

\section{Main results --- the infinite order energy functional}\label{sec:3}

In this section, we extend the ``inifinte order energy functional" approach to a general class of evolution equations \eqref{quadeqn} 
\[
u_t +(-\Delta)^\vt u= B(u,u), \qquad B(u,v)= R(Su\otimes Tu),
\] 
with  applications to several well-known examples.

We  consider \eqref{quadeqn} on a closed subspace $\sL \subset L^2(\Omega)$  which is invariant to the action of the Laplacian $\Delta$ and  of $B(u,v): \sL\times \sL \mapsto  \sL$.
The operators  $R, S, T$ are assumed to be homogeneous Fourier multipliers, i.e., they map one homogeneous
potential space to another  and  satisfy the estimates
\begin{gather}  \label{operassump}
\|Z w\|_{\hh^{\beta}} \le \kappa_{\beta} \|
w\|_{\hh^{\beta+\beta_Z}},\ Z \in \{R,S,T\}, \beta \in \R.
\end{gather}

We will assume that
\be \label{parameter}
\left\{\begin{array}{l}
\vt > \max\left\{\frac 23 \br +\frac{\bs+\bt}{3}, \frac 12 \br + \frac 12 \max\{\bs,\bt\}, \right.\\
\qquad \qquad \qquad \left. \frac 12 \max\{\bs,\bt\}, \frac{1}{4}\left[ \br+\bs+\bt + \frac d2\right]\right\} \\
\\
\vt < \min\{ \br + \min\{\bs,\bt\}+d, \br + \frac{\bs+\bt}{2}+ \frac d2 \}.
\end{array}\right.
\ee
The first condition on $\vt$ guarantees that  the nonlinear term is dominated by a ``sufficient amount" of dissipation, while the second is more technical in nature. Some of these requirements  can be circumvented
in some specific examples. 
Many models   in physics and biology are of the form \eqref{quadeqn} where the parameters satisfy
\eqref{parameter}, including the following prototypical cases; see Section \ref{sec:4} for details.   
\begin{itemize}
 \item[(i)] Burgers'
equation: Here $S=T=I, R=\partial_x$;  thus with $\bs=\bt=0, \br=1$.
\item[(ii)]
Navier-Stokes equations: Here $S=T=I$ and  $R={\mathbb P} \nabla$, where ${\mathbb P} $ is the Leray-Hopf projection on divergence free vector fields;  thus $\bs=\bt=0, \br=1$.
\item[(iii)]  The surface quasi-geostrophic
equation: Here $S=I, R=\nabla$ and $T=(-\cal{R}_2,\cal{R}_1)$ is the two-dimensional Riesz transform; thus $\bs=\bt=0, \br=1$.
\item[(iv)] The Keller-Segel model for chemotaxis: Here, $S=I, T=\nabla \Delta^{-1}, R=\nabla$ with $\bs=0, \bt=-1, \br=1$. 
\end{itemize}

\begin{theorem}   \label{thm:gensmalldata}
Set $\beta_c:=\beta_R + \beta_S + \beta_T + \frac{d}{2} - 2 \vt$. Let $u(\cdot)$ be the unique strong solution of \eqref{quadeqn}
subject to initial data $u_0$ such that  $\|u_0\|_{\hh^{\beta_c}}$ is sufficiently small. Assume that \eqref{parameter} holds. 
Then, there exists a choice of constants $\alpha_n>0, n=1, 2, \cdots $
such that the infinite order energy functional,
\[
{\cal E}(t)=\sum_{n=0}^\infty \alpha_n t^n\|(-\Delta)^{n\vt/2}u(t)\|_{\hh^{\beta_c}}^2, \quad \alpha_0=1
\]
is non-increasing  for all $t>0$.
In particular, ${\cal E}(t) \le {\cal E}(0)=\|u_0\|_{\hh^{\beta_c}}^2$ and we have the higher-order decay
\[
\|u(t)\|_{\hh^{n\vt+ \beta_c}}^2 \le \frac{1}{\alpha_nt^n}\|u_0\|_{\hh^{\beta_c}}^2.
\]
\end{theorem}

\begin{remark}\label{rem:largetime}
{\em In many applications, $0<\beta_c \le \vt$ and the nonlinear term is skew-symmetric,
 i.e., $(B(u,v),v)=0$. In this case, from$L^2$-integration of  \eqref{quadeqn} yields
\[
\|u(t)\|^2 + \int_0^t \|(-\Delta)^{\vt/2} u(s)\|^2 ds \le \|u_0\|^2.
\]
This implies  that 
$\|u(t)\|^2 \le \|u_0\|^2$ and $\displaystyle  \liminf_{t \ra \infty}\|\Lambda u\|^2=0$. By interpolation, we then have $\liminf_{t \ra \infty}\|u(t)\|_{\hh^{\beta_c}}=0$, i.e., the desired smallness condition in Theorem \ref{thm:gensmalldata} holds, at least  at certain late time $t\ge t_0>0$. It follows from Theorem \ref{thm:gensmalldata} that the modified energy functional 
\[
{\cal E}(t)=\sum_{n=0}^\infty \alpha_n (t-t_0)^n\|(-\Delta)^{n\vt/2}u(t)\|_{\hh^{\beta_c}}^2, \quad t \ge t_0,
\]
 is non-increasing in for all $t>t_0$ and in particular, ${\cal E}(t) \le \|u(t_0)\|^2_{\hh^{\beta_c}}$.\newline
The same result, with the same proof,  holds even if $\beta_c \le 0$, provided there exists $\gamma > \beta_c$ such that
$\sup_{t\ge 0}\|u(t)\|_{\hh^\g} < \infty$; see application to the 2D Navier-Stokes equation in section \ref{sec:2DNSE} as an example for this line of argument.}
\end{remark}
In many physically-relevant examples, it may be desirable to consider an $L^2-$based energy functional
as the $L^2-$norm represents energy.
 For simplicity, we consider the skew-symmetric case.
\begin{theorem}  \label{thm:main}
Consider the evolution equation \eqref{quadeqn}  with a skew-symmetric bi-linear form, $B(u,v)=R(Su\otimes Tv)$, with critical regularity of order $\beta_c:=\beta_R+\beta_S+\beta_T+\frac{d}{2}-2\vt$, such that \eqref{parameter} holds.
Let $u(\cdot)$ be a  strong solution of 
\eqref{quadeqn} on $(0,T)$ subject to $L^2$-initial data $u_0$ and assume it satisfies the following  decay  --- there exists a constant  possibly dependent on the initial data, $D_0=D(u_0)$  and $\bc < \g <2\vt$, such that 
\begin{gather}  \label{qdecay}
\sup_{0< t < T}{t}^{\frac{\g-\bc}{2\vt}} \|u\|_{\hh^\g} \le D_0. 
\end{gather}
Then, the infinite order energy functional (depending on the Kato-Ponce constants $C_n$'s \eqref{katoponce}),
\[
{\cal E}(t)=\sum_{n=0}^\infty \alpha_nt^n\|(-\Delta)^{n\vt/2}u(t)\|^2, \qquad 
\alpha_n=\left\{\begin{array}{cl} 1, & n=0,\\ \\ \displaystyle \frac{1}{C_nD_0^n}, & n \ge 1,\end{array}\right. 
\]
 is non-increasing for $0<t<T$. In particular, we have the high-order decay rate
\[
\|u(t)\|_{\dot{\h}^{n\vt}}^2 \le  \frac{1}{\alpha_nt^n}\|u_0\|^2, \qquad n=1,2,\ldots .
\]
\end{theorem}
\begin{remark}\end{remark}
\vspace*{-0.8pc}
\begin{itemize}
\item[(i)] In case $\|u_0\|_{\hh^{\beta_c}}$ is sufficiently small, it immediately follows from
Theorem \ref{thm:gensmalldata} that \eqref{qdecay} holds for any $\g >\beta_c$.
In certain cases, \eqref{qdecay} is satisfied also for large initial data in $\hh^{\beta_c}$ provided the initial data lies in more restrictive classes; see application to the Navier-Stokes equations in Section \ref{subsection:NSE}.
\item[(ii)] An appropriate modification of \cite{wei} or \cite{fk} shows that  a {\em mild} solution satisfying \eqref{qdecay} for some $T>0$ exists for {\em arbitrary} initial data $u_0 \in \hh^{\beta_c}$. The proof of existence of such mild solutions is sketched in the Appendix.
\end{itemize}

\section{Applications --- dissipation vs. quadratic nonlinearity}\label{sec:4}

In this section, we provide several applications of Theorems \ref{thm:gensmalldata} and \ref{thm:main}. Note that on the time intervals where the infinite order energy functional is non-increasing, the solution is smooth (as all higher derivatives are bounded).
 It is worthwhile to keep in mind that in some cases, it is known that the solution experiences a finite-time loss of regularity for large initial data (e.g., Keller-Segel model) or it is not yet known whether a globally regular solution exists for arbitrarily large initial data (the 3D Navier-Stokes equations). In these cases, either an  appropriate smallness assumption or regularity assumption must be made for our global regularity result to hold for time $t>0$.  Otherwise, we show that due to Remark \ref{rem:largetime}, regularity holds for large enough time, $t\geq t_0$.
 On the other hand, in cases such as  the viscous Burgers' or 2D Navier-Stokes  equations where it is well-known that  globally regular solutions exist for large classes of initial data, we show that the high-order decay rate stated in Theorem \ref{thm:main} is valid  for the corresponding large and only ``slightly'' smaller classes of initial data.

\subsection{Navier-Stokes equations}\label{subsection:NSE}
The incompressible Navier-Stokes (NS) equations  are given by 
\begin{gather*}
u_t - \Delta u + \nabla p + u\cdot \nabla\, u =0,\qquad \nabla \cdot u =0,
\end{gather*}
where $u:\R^d\times \R_+ \ra \R^d$ is the velocity vector field and $p$ is pressure.
The pressure can be regarded as a Lagrangian multiplier which imposes the divergence free condition. 
Due to the presence of pressure, these equations are nonlocal. It is customary to apply the Leray projection operator 
on the Navier-Stokes equations to eliminate pressure. In this case, they can be rewritten as
\begin{gather}  \label{nse}
u_t - \Delta u  + \p \,\nabla \cdot(u\otimes u) =0,\qquad \nabla \cdot u =0,
\end{gather}
where $\p$ is the Leray projection operator on divergence free vector fields. 
Here we have used the fact that $u$ is divergence free and in the absence of boundary, 
the Leray projection and the Laplacian commute. Note that (\ref{nse}) is of the form (\ref{quadeqn}) with $R=\p\, \nabla \cdot$ and $ T=S=I$.

It is well-known that when space dimension $d=2$, (\ref{nse}) admits globally regular (classical) solution. 
The question whether or not this is the case
when $d=3$, is still open.  Due to the work of Leray, it is well-known however that a 
weak solution of (\ref{nse}) is in fact regular for large times. Moreover, the 3D NS equations are locally well-posed for initial data 
$u_0 \in \h^{1/2}$ and a global regular solution exists in case the initial data
 $\|u_0\|_{\hh^{1/2}}$ is sufficiently small \cite{fk}.
We have the following results  concerning the Navier-Stokes equations,
 which is stated in terms of the $L^2$-based infinite order energy functional.

\subsubsection{3D Navier-Stokes equations}
\begin{theorem}  \label{thm:nse}
 Let $u$ be a Leray-Hopf weak solution of the Navier-Stokes equations on $[0,\infty)\times \R^3$
subject to initial data $u_0 \in  L^2(\R^3)$. Then the following hold.
\begin{itemize}

\item[(i)] There exist  constants, $C$, independent
of $u_0$, and $D_0=D(u_0)$, possibly dependent on $u_0$, such that for  sufficiently large  $t_0=t_0(u_0)>0 $, the
 modified infinite order energy functional,
\begin{gather}  \label{modifiedenergy}
{\cal E}(t) := \sum_{n=0}^{\infty} \alpha_n(t-t_0)^n\|(-\Delta )^{n/2}u(t)\|^2, \qquad 
\alpha_n :=\left\{\begin{array}{ll}1, & n=0, \\ \\ \displaystyle \frac{1}{C_nD_0^n}, & n>0, \end{array}\right.
\end{gather}
is  non-increasing for $t>t_0$. 
 In particular, we have the high-order decay estimate
\begin{gather} \label{nsedecay}
\|u(t)\|_{\hh^n}^2 \le \frac{1}{\alpha_n(t-t_0)^n}\|u_0\|^2, \quad t>t_0.
\end{gather}

\item[(ii)] Let  $u(\cdot)$ be the regular solution on $(0,\infty )$,
i.e., $u(\cdot) \in L^\infty_{loc}((0,\infty);\h^1)$, with $u_0 \in \h^{1/2}(\R^3)$.
Then there exists  constants $C$ and $D_0=D(u_0)$, such that  the
infinite order energy functional 
\begin{gather}  \label{unmodifiedenergy}
{\cal E}(t) := \sum_{n=0}^{\infty} \alpha_nt^n\|(-\Delta )^{n/2}u(t)\|^2 ,\qquad  
\alpha_n := \left\{\begin{array}{ll}1, & n=0,\\ \\\displaystyle \frac{1}{C_nD_0^n}, & n \ge 0, \end{array}\right.
\end{gather}
is non-increasing for all $t >0$.  In particular,  the estimate \eqref{nsedecay}  holds with $t_0=0$.

\item[(iii)] 
If $u(\cdot)$ decays at an exponential rate, $\|u(t)\| \lesssim e^{-\lambda t}$, then 
for each $n\ge1$, 
$\|u(t)\|_{\hh^n}$ decays to zero at the same exponential rate,  
$\|u(t)\|_{\hh^n} \lesssim e^{-\lambda t}$.
\end{itemize}
\end{theorem}

\subsubsection{2D Navier-Stokes equations}\label{sec:2DNSE}
 We recall that  the 2D NS equations, admit a globally
regular solution for initial data $u_0 \in L^2(\R^2)$ \cite{LR, CF, majb}.
In this case,  we have the following regularity result.
\begin{theorem}   \label{thm:nse2d}
Let $u$ be the regular solution of \eqref{nse} on $(0,\infty)$ subject to initial data $u_0 \in L^2(\R^2) \, \cap\,\hh^{- \beta }(\R^2),\, \beta \in (0,1)$. 

\begin{itemize}
\item[(i)] The infinite order energy functional  defined above in \eqref{unmodifiedenergy} is nondecreasing in $t$
and consequently,  \eqref{nsedecay} holds with $t_0=0$. 
\item[(ii)] If $u(\cdot)$ decays to zero at an exponential rate $\|u(t)\| \lesssim e^{-\lambda t}$, then 
for each $n\ge1$, 
$\|u(t)\|_{\hh^n}$ converges to zero at the same exponential rate, i.e., 
$\|u(t)\|_{\hh^n} \lesssim e^{-\lambda t}$.
\end{itemize}
 In particular,
if $u_0 \in L^p(\R^2) \cap L^2(\R^2), 1\le p <2$, all the conclusions above hold.
\end{theorem}

\begin{remark}\label{rem:1}
{\em 
Sharp high-order decay estimates \eqref{nsedecay} for large times, $t\gg 1$, were  obtained earlier in \cite{S5,SW,Ti}, under additional assumptions of initial integrability, e.g., $u_0 \in L^1 \cap L^2$ in \cite{S1,S5}, or algebraic decay, e.g.,
$ \|u(t)\|_{L^2} \lesssim (1+t)^{-\mu}$ in \cite{S5a} (the latter follows from the  former ---  $u_0 \in L^1 \cap L^2$ implies $L^2$ decay with $\mu=1/2$). 
In particular, under the  assumption of algebraic decay of  
$\|u(t)\|_{L^2}$, the high-order decay estimate  \eqref{nsedecay} with an  ``optimal'' constant (of the order of $n^n$) 
was derived in \cite{Ti}  using Gevrey class techniques \cite{FT}. 
The 2D estimate  \eqref{nsedecay} for $L^p(\R^2) \cap L^2(\R^2), 1\le p <2$ initial data can be found in 
 \cite{S5, Ti},  while the  decay result for 
$u_0 \in \hh^{- \beta} \cap L^2(\R^2)$ can be found in \cite{biswas}.    

Here, the high-order decay estimates \eqref{nsedecay} hold for general $L^2$ data and for all times,  as long as the solution remains regular  for  $t>t_0$. As before, our main focus is the  new approach  based on the use of an infinite-order
energy functional, which is independent of Fourier-based arguments (as in e.g., \cite{S5a}). This, in turn,  enables us to pursue a unified framework for analysis  the time decay of  a large  class of dissipative equations with quadratic nonlinearities.
} 
\end{remark}

\subsection{2D surface quasi-geostrophic equations}\label{subsec:QG}
The 2D surface 
quasi-geostrophic equation given by
\begin{gather}  \label{2dqg}
\eta_t + u \cdot \nabla \, \eta = -(-\Delta)^{\vt} \eta ,\ 0 < \vt \le 1,\quad u := (-{\cal R}_2,\cal{R}_1) \eta ,
\end{gather}
where ${\cal R}_i$ are the two-dimensional Riesz transforms, $\widehat{\cal R}_i(\xi)={\xi_i}/{|\xi|}$. 
This equation, which is of the form (\ref{quadeqn}) with $ R \mapsto I, S \mapsto {\cal R}, T \mapsto \nabla$, 
is an important model
in geophysical fluid dynamics and has received considerable attention recently; 
see for instance \cite{CW}, \cite{dong}  and the references therein. 
The subcritical and supercritical cases correspond to dissipation of order  $\frac 12 < \vt \le 1$ and $0 < \vt < \frac 12$ respectively.
The critical quasi-geostrophic equation, 
corresponding to $\vt = \frac 12$ is the two dimensional analogue of the 3D Navier-Stokes equations. 
The global well-posedness of this equation has been proven only recently \cite{cv, k}. We  focus here on the subcritical case.

\begin{theorem}  \label{thm:qg}
Let $\frac{2}{3} \le \vt \le 1,\ \delta_0>0$ and consider the solution $\eta$ of the 2D QG equation (\ref{2dqg}) subject to initial data 
$\eta_0 \in \h^{2-2\vt + \delta_0}(\R^2)$.
Then, there exist    constants  $D_0=D(\eta_0)$ and $C_n$ as in \eqref{katoponce},
such that the infinite order energy  functional,
\begin{gather*}  
{\cal E}(t) := \sum_{n=0}^{\infty} \alpha_nt^n\|(-\Delta )^{n\vt/2}u(t)\|_{L^2}^2, \qquad 
\alpha_n= \left\{\begin{array}{ll} 1, & n=0,\\ \\ \displaystyle \frac{1}{C_nD_0^n}, & n \ge 0.\end{array}\right.
\end{gather*}
  is non-increasing for all $t>0$.
Moreover, for  sufficiently large $t$, we also have $\|\eta(t)\|_{\hh^n}^2 = O(\|\eta(t)\|^2)$. 
In particular,
if $\eta(t)$ decays to zero at an exponential rate, $\|\eta(t)\| \lesssim e^{-\lambda t}$, then so are its spatial derivatives --- for each $n\ge 1$ 
$\|\eta(t)\|_{\hh^n} \lesssim e^{-\lambda t}$. 
\end{theorem}

\subsection{Keller-Segel model }\label{subsec:kellersegel}
We consider the Keller-Segel model,
\begin{gather}  \label{ksmodel}
\rho_t = \nabla \cdot \left(\rho u\right) + \Delta \rho,\qquad u = \nabla \Delta^{-1}\rho.
\end{gather}
This model is of the form (\ref{quadeqn}) with $\vt \mapsto 1, R \mapsto \nabla, S \mapsto I$ and $T \mapsto \nabla \Delta^{-1}$. It describes the collective motion of cells (usually bacteria or amoeba) that are attracted by a chemical substance 
and are able to emit it (see \cite{ksm}). Here $\rho $ is the cell concentration and $u$ is the drift velocity.
There has been a large amount of recent activity devoted to this model (see \cite{per1} and the references there in).
In particular, it was shown that the Keller-Segel equation admits a strong solution  if  $\|\rho_0\|_{L^{d/2}(\R^d)}$  
is sufficiently small, but on the other hand,
the solution experiences a finite-time blow-up (converges to Dirac delta) if the initial $L^d$-norm is larger than a critical value
(see \cite{per, horst} for $d=2$). In the framework of sufficiently small data, we have the following higher order smoothness result.
 \begin{theorem}   \label{thm:ksmodel}
Consider the $d$-dimensional Keller-Segel equation \eqref{ksmodel}, $d=2,3$, subject to sufficiently small initial data $\rho_0$ with $\|\rho_0\|_{\hh^{\frac{d}{2} -2}} \ll 1$.
 Then, there exists a global solution $\rho(\cdot,t)$ such that the  
  infinite order energy functional (corresponding to  \eqref{eqs:higherenergy} with $\vt=1$)
\[
{\cal E}(t) = \sum_{n=0}^{\infty} \alpha_nt^n\|(-\Delta)^{n/2}\rho(t)\|_{\hh^{\beta_c}}^2, \quad \alpha_0=1, \ \beta_c=\frac{d}{2}-2
\]
 is non-increasing for all $t$. In particular, we have the high-order decay estimate
\be
\|\rho(t)\|_{\hh^{n+\frac{d}{2}-2}}^2 \le \frac{1}{\alpha_nt^n}\|\rho_0\|^2_{\hh^{\frac{d}{2}-2}}
\ee
\end{theorem}

\begin{remark}\label{rem:3}
{\em 
The result of higher order decay for the Keller-Segel model is new.
For the case $d=3$, the critical space is $L^{3/2} \subset \hh^{-\frac 12}$, i.e., $\|\rho\|_{\hh^{-\frac 12}}\lesssim \|\rho\|_{L^{3/2}}$, and hence 
 our smallness assumption on $\|\rho_0\|_{\hh^{-\frac 12 }}$ is weaker than the usual smallness assumption on
$\|\rho\|_{L^{3/2}}$.
 On the other hand, in space dimension $d=2$, both the critical spaces $L^1$ and $\hh^{-1}$ are embedded in the homogeneous 
 Besov space $B_{\infty}^{-2,\infty}$. Since the embedding of $L^1$ in $\hh^{-1}$ is no longer true, our assumption
 on the smallness of $\|\rho\|_{\hh^{-1}}$ can be regarded as a different condition guaranteeing smoothness of solutions, in addition to
 the decay of their higher Sobolev norms. The new proof provided here involves only  ``energy techniques"; no use is made of the entropy-based estimates in e.g., \cite{per}.} 
 \end{remark}

\section{Proofs of main results}\label{sec:5}
\noindent
{\bf Proof of Theorem \ref{thm:gensmalldata}}. Corresponding to the general dissipation operator of order $\vt$, we set $\Lambda:=(-\Delta)^{\vt/2}$ so that $\|u\|_{\hh^{\beta}}=\|\Lambda^{\frac{\beta}{\vt}} u\|$. 
We will  need the following lemma. 
\begin{lemma}   \label{lem:gensmalldata}
Let $u$ be a solution of \eqref{1dburger} and assume that $\|u_0\|_{\hh^{\bc}}$ is sufficiently small. Then,
 for all $t>0,\, {\D \dt \|u\|_{\hh^{\bc}}^2 \le - \|\Lambda^{ \ffc+1} u\|^2}$ and 
\[
\|u(t)\|_{\hh^{\bc}}^2 \le 
\|u_0\|_{\hh^{\bc}}^2.
\]
\end{lemma}
\begin{proof}
Recall that $\|u\|_{\hh^{\bc}}^2=\|\Lambda^{\ffc}u\|^2$.
Taking $L^2-$inner product of \eqref{1dburger} with $\Lambda^{2\ffc} u$, we obtain for
any $0 \le \eps \le 1$
\bea
\frac{1}{2} \dt \|u\|_{\hh^{\bc}}^2 + \|\Lambda^{\ffc +1} u\|^2 &= & (\Lambda^{\ffc -\eps}R(Su\otimes Tu),\Lambda^{\ffc+\eps}u)\nn \\
&  \lesssim & \|\Lambda^{\ffc+\eps}u\|\|\Lambda^{\ffc +\ffr-\eps}(Su\otimes Tu)\|. \label{energyineq}
\eea
Due to \eqref{parameter}, there exists a choice of constants $\delta_0, \eps \in \R$, such that for 
\[
\zeta_0 := \bc +\br + \frac d2 -\eps \vt -\delta_0,
\] 
 the following inequalities are satisfied:
\[
\max\{\delta_0, \zeta_0\} < \frac d2,  \delta_0 +\zeta_0>0, \bc \le \delta_0+\bt \le \bc + \vt\  \mbox{and}\ 
\bc \le \zeta_0 + \bs \le \bc + \vt.
\]

 We will  need the following inequality for the homogeneous Sobolev norm of the product of two functions
(see \cite{ker}, \cite{rs}),  namely,
 \begin{gather}  \label{convineq}
 \|fg\|_{\hh^{\vt_1 + \vt_2 - \frac d2}} \le C\|f\|_{\hh^{\vt_1}}\|g\|_{\hh^{\vt_2}}, \qquad \vt_1 + \vt_2 >0, 
 \vt_i < \frac d2, i=1,2. 
 \end{gather}
Applying this inequality to \eqref{energyineq} with $\vt_1=\alpha_0, \vt_2=\zeta_0$ followed by interpolation, we obtain
\[
\|\Lambda^{\ffc+\eps}u\|\|\Lambda^{\ffc +\ffr-\eps}(Su\otimes Tu)\| \lesssim 
\|\Lambda^{\bc +1}u\|^2\|\Lambda^{\bc}u\|.
\]
Consequently,
\bes
\frac{1}{2} \dt \|u(t)\|_{\hh^{\bc}}^2 + \|\Lambda^{\ffc+1} u(t)\|^2(1- C\|u(t)\|_{\hh^{\bc}}) \le 0.
\ees
Thus, if $\|u_0\|_{\hh^{\bc}} < \sfrac{1}{2C}$, we conclude that
$\|u(t)\|_{\hh^{\bc}}^2$ is non-increasing for all $t>0$ and the lemma follows.
\end{proof}

We will now continue with the proof of the theorem. As before, taking inner product and differentiating, 
we obtain for $n \ge 1$, 
 \bea
&  &\dt \left[\alpha_n t^n\|\Lambda^nu(t) \|_{\hh^{\bc}}^2\right] \le  n
\alpha_nt^{n-1}\|\Lambda^nu\|_{\hh^{\bc}}^2 \nn \\
& & \quad - 2\alpha_nt^n\|\Lambda^{n+1}u\|_{\hh^{\bc}}^2 
+2\alpha_nt^n \|\Lambda^{n+1}u\|_{\hh^{\bc}}\|\Lambda^{n-1+\ffc}R(Su \otimes Tu)\|. \label{iteration} \\
& &\quad \le  (n +1)\alpha_nt^{n-1}\|\Lambda^nu\|^2 -
\alpha_nt^n\|\Lambda^{n+1}u\|^2 +
2\alpha_nt^n(1+\zeta)^{\frac{1+\zeta}{1-\zeta}}c_n^{\frac{2}{1-\zeta}}\|\Lambda^{\delta}u\|^{\frac{2}{1-\zeta}}\|\Lambda^nu\|^2
\nn \\
& & \quad \le \alpha_{n-1}t^{n-1}\|\Lambda^nu\|^2 -
\alpha_nt^n\|\Lambda^{n+1}u\|^2, \nn
\eea
Note that $\|\Lambda^{n-1+\ffc}R(Su \otimes Tu)\| \lesssim \|\Lambda^{n-1+\ffc+\ffr}(Su \otimes Tu)\|$.
For convenience, we will assume  that $n-1+\frac{\bc+\br}{\vt} \ge 0, n \ge 1$ (for those values of $n$ for which $n-1+\frac{\bc+\br}{\vt} < 0$, we may proceed as in proof of Lemma \ref{lem:gensmalldata}).
Applying now \eqref{katoponce}  followed by the Sobolev inequality, we obtain
\bea
\lefteqn{ 
\|\Lambda^{n-1+\ffc+\ffr}(Su \otimes Tu)\|}  \nn \\
& & \lesssim \|\Lambda^{n-1+\frac{\br+\zeta_0+\bs+\bc}{\vt}}u\|\|\Lambda^{\frac{\delta_0+\bt}{\vt}}u\|
+\|\Lambda^{n-1+\frac{\br+\zeta_0'+\bt+\bc}{\vt}}u\|\|\Lambda^{\frac{\delta_0'+\bs}{\vt}}u\|, \label{gennonlin}
\eea
where $\delta_0+\zeta_0=\frac d2,  \delta_0'+ \zeta_0'=\frac d2$  and they moreover satisfy
\be  \label{morechoices}
\left\{\begin{array}{l}
 \bc < \delta_0+\bt<2\vt, \bc < \br+\bs+\zeta_0<2\vt, 0<\delta_0<\frac d2,\\
\\
 \bc < \delta_0'+\bs<2\vt, \bc < \br+\bt+\zeta_0'<2\vt, 0<\delta_0'<\frac d2.
\end{array}\right.
\ee
Such a choice of  constants $\delta_0, \delta_0', \zeta_0, \zeta_0'$ is possible thanks  to \eqref{parameter}.
Due to \eqref{morechoices}, it is possible to choose
 $\bc \le \g <\min\{\delta_0+\bt, \delta_0'+\bs\} $ such that
\[
\zeta := 1- \frac{\g -\bc}{\vt} > \max\{0, \frac{\br+\bs+\zeta_0}{\vt}-1, 
\frac{\br+\bt+\zeta_0'}{\vt}-1\}.
\]
 Using \eqref{gennonlin}, \eqref{iteration}, interpolation and subsequently, using Young's inequality, we obtain
\beas
&  &\dt \left[\alpha_n t^n\|\Lambda^nu\|_{\hh^{\bc}}^2\right] \le  \nn \\
& & \qquad n
\alpha_nt^{n-1}\|\Lambda^nu\|_{\hh^{\bc}}^2  - 2\alpha_nt^n\|\Lambda^{n+1}u\|_{\hh^{\bc}}^2 
+2c_n\alpha_nt^n \|\Lambda^{n+1}u\|_{\hh^{\bc}}\|\Lambda^{\frac{\g}{\vt}}u\|
\|\Lambda^{n+\zeta}u\|_{\hh^{\bc}} \\
& & \qquad \le \alpha_nt^{n-1}\|\Lambda^nu\|_{\hh^{\bc}}^2  - 2\alpha_nt^n\|\Lambda^{n+1}u\|_{\hh^{\bc}}^2 
+2c_n\alpha_nt^n \|\Lambda^{n+1}u\|_{\hh^{\bc}}^{1+\zeta}\|\Lambda^{\frac{\g}{\vt}}u\|
\|\Lambda^n u\|_{\hh^{\bc}}^{1-\zeta}\\
& & \qquad \le \alpha_nt^{n-1}\|\Lambda^nu\|_{\hh^{\bc}}^2  - 
\alpha_nt^n\|\Lambda^{n+1}u\|_{\hh^{\bc}}^2 + 2\alpha_nt^n(1+\zeta)^{\frac{1+\zeta}{1-\zeta}}c_n^{\frac{2}{1-\zeta}}\|\Lambda^{\frac{\g}{\vt}}u\|^{\frac{2}{1-\zeta}}\|\Lambda^nu\|_{\hh^{\bc}}^2.  
\eeas
By interpolating $\|\Lambda^{\frac{\g}{\vt}}u\|$ in terms of $\|u\|_{\hh^{\bc}}$ and 
$\|\Lambda u\|_{\hh^{\bc}}$ and proceeding exactly as in the proof of Theorem \ref{thm:1dburger}, we are done.
\eop

\medskip\noindent
{\bf Proof of Theorem \ref{thm:main}.} The proof is similar to that of Theorem \ref{thm:gensmalldata} and
Theorem \ref{thm:1dburgerl2}.\\[5pt]

\noindent
{\bf Proof of Theorem \ref{thm:nse}.} Here  $A=-\Delta$ and $\Lambda=(-\Delta)^{1/2}$ corresponding to  $\vt=1$, and $R=\p \nabla \cdot$, where $\p$ is the Leray-Hopf projection operator on divergence free vector fields. 
Note that $R$ is a pseudodifferential operator of order one and satisfies (\ref{operassump}) with $\br =1$. Moreover, we let $
T=S=I$. This choice of operators yield  $ \bt=\bs=0$
and $\bc = \frac 12$.

We first prove part (i).  
From Remark \ref{rem:largetime}, it immediately follows that there exists a  
regular solution that satisfies the estimate $\sup_{t> t_0} (t-t_0)^{\frac{1}{2}(\g - 1/2)}\|u(t)\|_{\hh^\g}
< \infty $ and  moreover, it coincides with the weak solution (see \cite{CF}) for all $t\ge t_0$.
  Applying Theorem \ref{thm:main}, the claim immediately follows.

\noindent
To prove part (ii), note that the decay condition in (\ref{qdecay}) translates into
\begin{gather}  \label{nsedecaycond}
t^{\frac{1}{2}(\g - \frac 12)}\|\Lambda^{\g}u(t)\| \le D(u_0),\quad  \frac{1}{2}< \g < 2.
\end{gather}
Now,  since $u(t)$ is a regular solution on $(0,T)$ for any $T>0$ and $u_0 \in \hh^{\frac{1}{2}}$,  
for any $\g > \frac{1}{2}$ and $\delta >0$, we have (see \cite{fk}) that $\sup_{t \in [\delta , T]} 
\|u (t)\|_{\hh^{\g}} < \infty $ 
and consequently, \eqref{nsedecaycond} holds on $[\delta , T]$. To complete  the proof of part (ii), 
we only need to show \eqref{nsedecaycond} for $t \in [0,\delta]\cup [T, \infty)$.
 Let $t_0$ as defined in part (i) and $T = t_0+1$. Then by Theorem \ref{thm:localexistence}, 
for all $t >t_0$, we have 
${\D  (t-t_0)^{\frac{1}{2}(\g - \frac 12)} \|u(t)\|_{\hh^\g} \le 2 \|u(t_0)\|_{\hh^{1/2}} < \epsilon }$. 
Thus, noting
that ${\D \sup_{t \in [t_0+1,\infty)}\frac{ t^{1/2}}{(t-t_0)^{1/2}} < \infty}$, it follows that
${\D \sup_{t\in [t_0+1,\infty)} t^{\frac12(\g - \frac 12)}\|u(t)\|_{\hh^\g}  , \infty }$.
The requisite condition follows
for $t \in [T, \infty)$ by part (i). For $t \in [0,\delta]$, it follows from Theorem \ref{thm:localexistence} 
provided $\delta $ is sufficiently small.

\noindent
Finally, we  prove part (iii). Let $t_0$ be as defined in part (i). 
By Theorem \ref{thm:localexistence}, for $t \in [t_0, \infty)$, the weak solution $u(t)$ is in fact unique and strong and satisfies
$\|u(t)\|_{\hh^{1/2}} < \epsilon/2 $. Thus, for any $t \in [t_0+1,\infty )$, 
we can apply Theorem \ref{thm:gensmalldata}
 with initial data $u(t-1)$ to obtain 
${\D \sup_{s\in [0,1]}s^{\frac12(\g-1/2)}\|\Lambda^{\g - 1/2}u(s+t)\|< \epsilon }$, and Theorem \ref{thm:main} implies
${\D \|u(t)\|_{\hh^n} \le C_n\|u(t-1)\|}$. This completes the proof. \eop \\[2pt]

\noindent
{\bf Proof of Theorem \ref{thm:nse2d}}. 
The statement in the theorem concerning $u_0 \in L^p(\R^2) \cap L^2(\R^2)$ follows immediately from the first
part in view of  the inequality, e.g., \cite{Chemin}
\begin{gather*}
 \|u_0\|_{\hh^{-\beta}}\le \|u_0\|_{L^p}, \qquad \beta = 2\left(\frac{1}{p}-\frac{1}{2}\right), 
\ 1<p\le 2. 
\end{gather*}
We will now prove the remainder of the theorem.
In this case, $\bc=0$ and 
by Theorem \ref{thm:main}, it is enough to establish
\begin{gather}  \label{2dnsedecaycond}
\sup_{t\in(0,\infty)} t^{\g/2} \|u(t)\|_{\hh^{\g}} < \infty\  \mbox{for some}\  0 < \g < 1.
\end{gather}
We first claim  that it is enough to establish
\begin{gather}  \label{nsel2decaycond}
\liminf_{t \ra \infty} \|u(t)\| =0.
\end{gather}
Indeed,
if  (\ref{nsel2decaycond}) holds then there exists $t_0 >0$ such that 
$\|u(t_0)\| < \epsilon $ where $\epsilon $ is as 
in Theorem \ref{thm:localexistence}. Thus, by Theorem \ref{thm:localexistence}, we have
$ \sup_{t\in(t_0,\infty)}(t-t_0)^{\frac{\g}{2}} \|u(t)\|_{\hh^{\g}} < \infty $. Recall  that if the 2D NS solutions satisfy  
$\|u(\epsilon)\|_{\h^{\g}} < \infty $ for some $0< \gamma <1$, then for any later time $T > \epsilon $, we have
$\sup_{[\epsilon ,T]}\|u(t)\|_{\h^\g} < \infty $. Using these two facts as well as the local result (near $t=0$)
in Theorem \ref{thm:localexistence}, 
and proceeding as in the proof of part (ii) of Theorem
\ref{thm:nse}, one can now easily obtain (\ref{2dnsedecaycond}). \comments{Note  that 
in case $u_0 \in L^1(\R^2) \cap L^2(\R^2)$, (\ref{nsel2decaycond})
follows  from the fact that  $\|u(t)\| = O(\frac{1}{\sqrt t})$ (see \cite{S1}).}\newline
We  now turn to  prove (\ref{nsel2decaycond}) as well as the $L^2-$decay 
\begin{gather}  \label{2dnseenergydecay}
\|u(t)\|^2 = O\left( t^{-\frac{\beta}{1+\beta}}\right),
\end{gather}
for initial data  $u_0 \in \hh^{-\beta} \cap L^2(\R^2),\ \beta \in (0,1)$. 
Arguing along the lines  of Theorem \ref{thm:nse}, the energy inequality implies that 
\begin{gather*}
\liminf_{t\ra \infty} \|u(t)\|_{\hh^1} =0.
\end{gather*}
Therefore, if we can establish 
\begin{gather}  \label{unifbd1}
\sup_{t \in [0, \infty)}\|u(t)\|_{\hh^{-\beta}} <\infty, 
\end{gather}
then  (\ref{nsel2decaycond})  follows by interpolation. Subsequently, one can also use (\ref{unifbd1}) and the conclusion (\ref{nsedecay}) 
(with $t_0=0$ and $n=1$) to establish (\ref{2dnseenergydecay}).\newline
 To establish (\ref{unifbd1}),  we estimate  the nonlinear term: fix $0 < \epsilon < \beta $; then with   $A= -\Delta$ we have,
\begin{align}
  |(B(u,u), A^{-\beta}u)| &= |(A^{-\frac{1+\beta}{2} - \frac{\epsilon}{2}}B(u,u), 
 A^{\frac{1-\beta}{2} + \frac{\epsilon}{2}}u)| \nn \\
&  \le C\|A^{- \frac{\beta}{2}}u\|\|A^{\frac 12 -  \frac{\epsilon}{2}}u\|\|A^{\frac{1-\beta}{2}  + \frac{\epsilon}{2}}u\| 
\le C\|A^{- \frac{\beta}{2}}u\|\|A^{\frac 12}u\|\|A^{\frac{1-\beta}{2}}u\|. \label{impbilinineq}
\end{align}
To obtain the first inequality in (\ref{impbilinineq}), 
we  note that $B(u,u)= \nabla \cdot (u\otimes u )$ and then use (\ref{convineq}) 
with $d=2, \vt_1 = - \beta $ and $ \vt_2= 1 - \epsilon$.  
The last inequality in (\ref{impbilinineq})  is obtained using interpolation.\newline
 Multiplying (\ref{nse}) 
by $A^{- \beta}u $ and integrating (in space variables) we obtain
\begin{align*}
 \frac 12 \dt \|A^{-\frac{\beta}{2}}u(t)\|^2 &+\|A^{\frac{1-\beta}{2}}u(t)\|^2  \le
 |(B(u,u), A^{-\beta}u)|  \\
 &  \le C\|A^{- \frac \beta 2}u\|\|A^{\frac{1-\beta}{2}}u\|\|A^{\frac 12}u\| \le \frac 12\|A^{\frac{1-\beta}{2}}u\|^2 + 
C \|A^{- \frac{\beta}{2}}u\|^2\|A^{\frac 12}u\|^2;
 \end{align*}
here, the first inequality follows from (\ref{impbilinineq}) and the second from Young's inequality.
Consequently, we have
\begin{gather*}
\dt \|A^{-\frac{\beta}{2}}u(t)\|^2 - C\|A^{- \frac{\beta}{2}}u(t)\|^2\|A^{\frac 12}u\|^2 \le 0.
\end{gather*}
Applying Gronwall's inequality and recalling that $\|A^{-\frac{\beta}{2}}u\|^2 = \|u\|^2_{\hh^{- \beta}}$, we immediately obtain
\begin{gather*}
\|u(t)\|^2_{\hh^{- \beta}} \le \exp \left(C\int_0^t \|A^{1/2}u(s)\|^2 \, ds\right)\|u_0\|^2_{\hh^{- \beta}}
\le  \exp\left(C\|u_0\|^2\right)\|u_0\|^2_{\hh^{- \beta}}.
\end{gather*}
The last inequality on the right follows from the well-known Leray energy inequality. This proves (\ref{unifbd1}). \eop \\[2pt]

\noindent
{\bf Proof of Theorem \ref{thm:qg}.} 
Here we take $A= (-\Delta)^{\vt}, R=I, T={\cal R}, S=\nabla$
in Theorem \ref{thm:main}. Thus, $\br =0, \bs=1$ and $\bt=0$.
Theorem \ref{thm:main} now implies that if for some 
$\delta, \  \frac{2}{\vt}-2 < \delta < 2 $, the following condition holds,
\begin{gather}  \label{quasidecay}
 t^{\frac{\delta}{2} - \frac12(\frac{2}{\vt}-2)}\|\eta(t)\|_{\hh^{\vt \delta}} \le D(\eta_0),
\end{gather}
then we are done. Note first 
that due to Theorem 2.1 in \cite{CW}, ${\D \sup_{[0,t_0)}\|\eta\|_{\hh^{2-2\vt + \delta_0}}
< \infty }$ for any $t_0 \ge 0$. Moreover, $\bc=2-2\vt \le \vt$ for $\frac 23 \le \vt \le 1$. 
Invoking Remark \ref{rem:largetime} and proceeding as in the proof of Theorem \ref{thm:nse2d}, we see that
\eqref{quasidecay} holds. The remainder of the proof is similar. \eop \\[2pt]


\noindent
{\bf Proof of Theorem \ref{thm:ksmodel} } This follows immediately from Theorem \ref{thm:gensmalldata}.

\section{Appendix}
Here we show that a mild solution of \eqref{quadeqn} satisfying the deacy assumption \eqref{qdecay} exists 
 locally in time; moreover, if the initial data  is sufficiently small in appropriate homogeneous Sobolev space, then this mild solution persists  globally in time.
For the special case of the Navier-Stokes equation, the theorem below was first proven by Fujita and Kato \cite{fk}. 
 We will sketch the proof for completeness. 
\begin{theorem}  \label{thm:localexistence}
Consider the evolution equation \eqref{quadeqn} with ``critical" order of regularity  $\beta_c:=\beta_R+\beta_S+\beta_T+\frac{d}{2}-2\theta$, subject to initial conditions $u_0 \in \hh^{\beta_c}$. Assume that 
\begin{gather}   \label{parameter2}
\frac{1}{2}\left\{\beta_R + \max \{\beta_S, \beta_T\}\right\} < \vt <  
\beta_R + \frac{\beta_S+\beta_T}{2} + \frac d2.
\end{gather}
 \comments{ and 
\begin{gather}  \label{gdef}
\max\{ \beta_c , \frac{\beta_S+\beta_T}{2}\} < \g < 
\min\left\{ \min\{\beta_S,\beta_T\} +\frac d2 , \beta_c + \vt\, \right\}.
\end{gather}
}
Then,  there exists a  classical solution of   \eqref{quadeqn},   $u(\cdot,t), \ t\in (0,T)$
which belongs to the class $C([0,T]; \hh^{\beta_c}) \cap C((0,T); \hh^\g)$ and satisfies
\eqref{qdecay}, for an adequate $\bc < \g < \bc+\vt$. 
Moreover, there exists an $\epsilon >0$ independent of the initial data $u_0$, such that
if $\|u_0\|_{\hh^{\beta_c}} < \epsilon $, then there exists a strong solutions global in time $u \in C([0,\infty); \hh^{\beta_c})$, and the following estimate holds
\begin{gather*}
\sup_{t \in (0,\infty)}\max \left\{\|u(t)\|_{\hh^{\beta_c}},  t^{\frac{\g-\beta_c}{2\vt}} \|u(t)\|_{\hh^\g}\right\}
\le 2\|u_0\|_{\hh^{\beta_c}}.
\end{gather*}
\end{theorem}

As an example, the last the last theorem applies to Burgers' equation \eqref{1dburger} with $\beta_c=-\sfrac{1\!}{2}$ and $\vt=1$ (so that \eqref{parameter2} holds $\frac12< \vt< \frac32$), and high-order decay  follows, $t^{\frac{2\gamma+1}{4}}\|u(t)\|_{\hh^\gamma}$ with $\gamma\in (-\frac12,\frac12)$. 
\begin{proof}
  The proof of this result follows the method of \cite{wei} (see also \cite{bisw}) for the Navier-Stokes equations. 
 We will use fixed point method to obtain a the mild solution
 of (\ref{quadeqn}), namely,
 \begin{gather}  \label{mildsol}
 u(t) = e^{-tAu_0} + \int_0^t e^{-(t-s)A}B(u(s),u(s))\, ds.
 \end{gather}
 Fix any $0 < T \le \infty $ and note that due to (\ref{linest}), for any $\beta \in \R$, it follows that
 \begin{gather}  \label{lindecay}
 \|e^{-tA}u_0\|_{\hh^{\beta}} \le C\|u_0\|_{\hh^{\beta}}\ \mbox{and}\ 
 t^{\frac{\g -\beta}{2\vt}} \|e^{-tA}u_0\|_{\g} \lesssim \|u_0\|_{\hh^{\beta}}, 0 < t < T, \g > \beta.
 \end{gather}
 Let $\g > \beta_c$ be fixed. Define
 \begin{gather}   \label{mdef}
 M(T) = M := \sup_{t \in (0,T)} 
  t^{\frac{\g -\beta_c}{2\vt}} \|e^{-tA}u_0\|_{\g}.
 \end{gather}
 It is easy to see that $M(T) \ra 0$ as $T\ra 0$. To see this, simply note that given any $\delta >0$, there exists
 $u_0' \in \hh^{\g}$ and such that $\|u_0 - u_0'\|_{\hh^{\beta_c}} < \delta $ 
 and by  (\ref{lindecay}), for $0 < t <T$, we have
 \begin{gather*}
 t^{\frac{\g -\beta_c}{2\vt}}\|e^{-tA}u_0\|_{\hh^\g} \le \|u_0 - u_0'\|_{\hh^\g}+
 t^{\frac{\g -\beta_c}{2\vt}} \|e^{-tA}u_0'\|_{\g} \le t^{\frac{\g -\beta_c}{2\vt}}\|u_0'\|_{\g}.
 \end{gather*}
 The first term in the right hand side of the above inequality is less than $\delta $ while the second approaches zero as $T\ra 0$.

 \medskip\noindent
 Consider the linear Banach space 
\begin{align*}
  \V = \left\{u \in  C([0,T];\hh^{\beta_c})\cap C((0,T); \hh^{\g}): 
 \|u\|_\V:= \sup_{0< t < T}\max\{ \|u(t)\|_{\hh^{\beta_c}}, 
  t^{\frac{\g - \beta_c}{2\vt}}\|u(t)\|_{\hh^\g}\}< \infty \ \right\},
 \end{align*}
and let $E \subset \V$ be the closed ball,
 \begin{gather}  \label{edef}
 E:= \{u \in \V: \|u\|_\V \le 2M\}.
 \end{gather}
For $u,v \in \V$, we define  
 \begin{gather*}  
 S(u,v):= \int_0^t e^{-(t-s)A}B(u(s),v(s))\, ds,
 \end{gather*}
 and claim that $S(\cdot,\cdot):\V \times \V \ra \V$ is a bounded bilinear operator, i.e., 
 \begin{gather}  \label{bddbilin}
 \|S(u,v)\|_\V \lesssim \|u\|_\V\|v\|_\V. 
 \end{gather}
Indeed, note  that due to (\ref{operassump}), (\ref{lindecay}) and (\ref{convineq}), we have
 \begin{align*}
 \|e^{-(t-s)A}B(u,v)\|_{\hh^\g} & \lesssim {(t-s)^{-\frac{\beta_c+2\vt -\g}{2\vt}}}\|u(s)\|_{\hh^\g}\|v(s)\|_{\hh^\g}
  \lesssim (t-s)^{-\frac{(\beta_c+2\vt -\g)}{2\vt}}s^{-\frac{\g -\beta_c}{\vt}}\|u\|_\V\|v\|_\V.
 \end{align*}
 Using this and the elementary inequality ${\D \int_0^t \frac{1}{(t-s)^as^b}\, ds \le t^{1-a-b}, 0 < a,b <1}$, one obtains
 \begin{gather*}
  t^{\frac{\g-\beta_c}{2\vt}}\|S(u,v)\|_{\hh^\g} \lesssim \|u\|_\V\|v\|_\V.
 \end{gather*}
 The other piece of the norm can similarly estimated. The rest of the proof is now standard. 
 One defines a map $\tau:\V \ra \V$ by the formula $\tau u = e^{-tA}u_0 + S(u,u)$. Using the estimates, one can show that
 it is a contractive self map of $E$ if $M$ is sufficiently small. From the fact that $M(T) \ra 0$ as $T \ra 0$ or (\ref{lindecay}), this holds
 if either $T$ is small enough or $\|u_0\|_{\hh^{\beta_c}}$ is sufficiently small.
\end{proof}

\end{document}